\documentclass[12pt,reqno]{amsart}
\usepackage[english]{babel}
\usepackage{a4wide}
\usepackage[T1]{fontenc}
\usepackage{amssymb,amsmath,amsthm,latexsym,ulem}
\usepackage{mathrsfs}
\usepackage[usenames,dvipsnames]{color}
\usepackage{graphicx}
\usepackage{mdwlist}
\usepackage{enumerate}
\usepackage{mathtools,dsfont,wasysym}
\usepackage{stmaryrd}
\usepackage{hyperref}
\usepackage{multicol}
\hypersetup{colorlinks=true,  linkcolor=blue, citecolor=red, urlcolor=red}

\newtheorem{theorem}{Theorem}[section]

\newtheorem{proposition}[theorem]{Proposition}

\newtheorem{corollary}[theorem]{Corollary}

\theoremstyle{definition}
\newtheorem{definition}[theorem]{Definition}
\newtheorem{example}[theorem]{Example}

\theoremstyle{remark}
\newtheorem{remark}[theorem]{Remark}
\numberwithin{equation}{section}

\newcommand{\R}{\mathbb{R}}

\newcommand{\N}{\mathbb{N}}

\newcommand{\dt}{\delta}

\makeatletter
\newcommand{\markthis}[3]{
  \overset{
    \textup{\makebox[0pt]{#1}}%
    \def\@currentlabel{#1}%
    \ltx@label{#2}%
  }{
    #3%
  }%
}

\DeclareMathOperator{\dist}{dist\,}

\def\cantor{\mathfrak{C}}

\newcommand{\nn}[1]{{\left\vert\kern-0.25ex\left\vert\kern-0.25ex\left\vert #1 
		\right\vert\kern-0.25ex\right\vert\kern-0.25ex\right\vert}}
\renewcommand{\geq}{\geqslant}
\renewcommand{\leq}{\leqslant}

\newcommand{\Free}{{\mathcal F}}

\newcommand{\Lip}{{\mathrm{Lip}}_0}

\newcommand{\eps}{\varepsilon}

\newcounter{smallromans}

	{\end{list}}

\makeatletter
\renewcommand{\tocsection}[3]{%
	\indentlabel{\@ifnotempty{#2}{\bfseries\ignorespaces#1 #2\quad}}\bfseries#3}
\renewcommand{\tocsubsection}[3]{%
	\indentlabel{\@ifnotempty{#2}{\ignorespaces#1 #2\quad}}#3}

\newcommand\@dotsep{4.5}
\def\@tocline#1#2#3#4#5#6#7{\relax
	\ifnum #1>\c@tocdepth 
	\else
	\par \addpenalty\@secpenalty\addvspace{#2}%
	\begingroup \hyphenpenalty\@M
	\@ifempty{#4}{%
		\@tempdima\csname r@tocindent\number#1\endcsname\relax
	}{%
		\@tempdima#4\relax
	}%
	\parindent\z@ \leftskip#3\relax \advance\leftskip\@tempdima\relax
	\rightskip\@pnumwidth plus1em \parfillskip-\@pnumwidth
	#5\leavevmode\hskip-\@tempdima{#6}\nobreak
	\leaders\hbox{$\m@th\mkern \@dotsep mu\hbox{.}\mkern \@dotsep mu$}\hfill
	\nobreak
	\hbox to\@pnumwidth{\@tocpagenum{\ifnum#1=1\bfseries\fi#7}}\par
	\nobreak
	\endgroup
	\fi}
\AtBeginDocument{%
	\expandafter\renewcommand\csname r@tocindent0\endcsname{0pt}
}
\def\l@subsection{\@tocline{2}{0pt}{2.5pc}{5pc}{}}
\makeatother

\usepackage{todonotes}

\begin{document}

	\title[SSD points in Lipschitz-free]{On the strongly subdifferentiable points in Lipschitz-free spaces}

 \dedicatory{Dedicated to Professor William B. Johnson on the occasion of his 80th birthday.}

		\author[Cobollo]{Christian Cobollo}
	\address[Cobollo]{Universitat Politècnica de València. Instituto Universitario de Matemàtica Pura y Aplicada, Camino de Vera, s/n 46022 Valencia, Spain \newline
		\href{http://orcid.org/0000-0000-0000-0000}{ORCID: \texttt{0000-0002-5901-5798} }}
	\email{\texttt{chcogo@upv.es}}

\author[Dantas]{Sheldon Dantas}
\address[Dantas]{Department of Mathematical Analysis, University of Granada, E-18071 Granada, Spain. \newline
	\href{https://orcid.org/0000-0001-8117-3760}{ORCID: \texttt{0000-0001-8117-3760}}}
\email{\texttt{sheldon.dantas@ugr.es}}

	\author[H\'ajek]{Petr H\'ajek}
\address[H\'ajek]{Department of Mathematics\\Faculty of Electrical Engineering\\Czech Technical University in Prague\\Technick\'a 2, 166 27 Prague 6\\ Czech Republic \newline \href{https://orcid.org/0000-0002-1714-3142}{ORCID: \texttt{0000-0002-1714-3142}}}
\email{\texttt{hajekpe8@fel.cvut.cz}}

	\author[Jung]{Mingu Jung}
\address[Jung]{June E Huh Center for Mathematical Challenges, Korea Institute for Advanced Study, 02455 Seoul, Republic of Korea \newline
\href{http://orcid.org/0000-0003-2240-2855}{ORCID: \texttt{0000-0003-2240-2855} }}
\email{\texttt{jmingoo@kias.re.kr}}

	\begin{abstract} In this paper, we present some sufficient conditions on a metric space $M$ for which every molecule is a strongly subdifferentiable (SSD, for short) point in the Lipschitz-free space $\mathcal{F}(M)$ over $M$. Our main result reads as follows: if $(M,d)$ is a metric space and $\gamma > 0$, then there exists a (not necessarily equivalent) metric $d_{\gamma}$ in $M$ such that every finitely supported element in $\Free(M, d_{\gamma})$ is an SSD point. As an application of the main result, it follows that if $M$ is uniformly discrete and $\eps > 0$ is given, there exists a metric space $N$ and a $(1+\eps)$-bi-Lipschitz map $\phi: M \rightarrow N$ such that the set of all SSD points in $\Free(N)$ is dense.
	\end{abstract}

	\thanks{ }
	
	
	\maketitle

	\thispagestyle{plain}


\section{Introduction}
We focus our attention on a concept of differentiability of the norm of a Banach space, which is called \textit{strong subdifferentiability} (SSD, for short). To put it simply, studying when the norm of a Banach space is strongly subdifferentiable is essentially the same as investigating what is lacking for a Gâteaux differentiable norm to become a Fréchet differentiable norm. In fact, SSD is a strictly weaker concept than Fréchet differentiability (indeed, every point of a finite-dimensional Banach space is SSD). More precisely, an element $x \in X$ is an SSD point if the limit of the expression $\frac{\|x+th\| - \|x\|}{t}$ exists uniformly in $h$ as $t \rightarrow 0^+$. To give the interested reader a sense of what is known about SSD in classical Banach spaces, we present some relevant results in this area. In what follows, we say that a Banach space $X$ is SSD when {\it every} point of its unit sphere is an SSD point. Every uniformly smooth Banach space is SSD since its norm is uniformly Fréchet differentiable on the unit sphere, and if a dual space $X^*$ is SSD, then $X$ must be reflexive (see \cite[Theorem~ 3.3]{FP}) which implies in particular that neither $\ell_1$ nor $\ell_\infty$ are SSD. On the other hand, if $X$ is a predual of a Banach space with the $w^*$-Kadec--Klee property, then $X$ is SSD (see \cite[Proposition~2.6]{DKLM}); hence $c_0$ is an example of a non-reflexive SSD Banach space. It is well-known that a Banach space admitting a Fr\'echet differentiable norm is an Asplund space. In fact, it turns out that the same implication holds for a Banach space admitting an SSD renorming---see \cite{FP, GMZ}. 
However, one of the important distinctions between Fr\'echet differentiability and SSD consists in the fact that the set of Fr\'echet differentiability points is always $G_\delta$ while the set of SSD points is an $F_{\sigma \delta}$ set (cf. \cite[Proposition 5]{GMZ}). In particular, if the set of points of Fr\'echet differentiability is dense, then it is automatically $G_\delta$ dense, which is not true for SSD points. Indeed, there is an example of Banach space, where the set of SSD points is both dense and meager (e.g., SSD points in $\ell_1$ are nothing but finitely supported elements as observed in \cite[Example 1.1]{GGS} and \cite[Theorem 6]{F}); so it cannot be residual. Let us also mention that any separable non-Asplund space (i.e., separable spaces whose dual is not separable)  admits an equivalent norm which is nowhere SSD except at the origin (see \cite[Theorem 1]{GMZ} and \cite[Thm III.1.9 and Prop.III.4.5]{DGZ}), and for every space with a fundamental biorthogonal system, there exists a dense subspace with an SSD norm---we send the reader to \cite{DHR, DHR1} for more information in this line of research. 



\vspace{0.2cm}

For a comprehensive study of SSD, we recommend consulting \cite{DGZ, FP}, along with a recent monograph \cite{GMZ22}. Additionally, we suggest the following references \cite{AOPR, Contreras, CP, DJMazR, F, GGS, Godefroy, GMZ, Gregory} for insights into this theory across various Banach spaces and \cite{FHHMZ} for the basic theory of Banach spaces.

\vspace{0.2cm}

Here, our focus lies in exploring the presence of SSD points within Lipschitz-free spaces over certain metric spaces $M$. It is important to note that the existence of Fréchet differentiable points in $\Free(M)$ relies on the boundedness and uniform discreteness of $M$ (see \cite{BLR} and, more generally, \cite[Theorem 4.3]{PRZ}). We will observe a significant change in this context concerning SSD points (cf. Corollary \ref{main}). On a different note, Aliaga and Rueda Zoca demonstrated that within a bounded uniformly discrete metric space $M$, a finitely supported element $\mu \in S_{\Free(M)}$ is Fréchet differentiable if and only if it is Gâteaux differentiable. They obtained such an equivalence by leveraging a geometric condition on the pairs $(x_i, y_i)$ involved in an optimal representation of $\mu=\sum_{i=1}^n \lambda_i m_{x_i,y_i}$ (cf. \cite[Theorem 3.5]{AR}). A similar result holds true for SSD and Fréchet differentiability points (cf. Proposition \ref{Prop:unifdiscrete}) when focusing on the Lipschitz-free space over $\cup_n [x_n, x_{n+1}]$, where the element is of the form $\mu = \sum_n a_n m_{x_n, x_{n+1}}$ for some $\{x_n\}_n \subseteq M$ containing $0$ and $\sum_n a_n = 1$ with $a_n > 0$.

\vspace{0.2cm}

As we have already mentioned some of our results, let us now present the contents of this manuscript in a more systematic manner. In the next section, we briefly present the necessary background on Lipschitz-free spaces and SSD points. The third (and final) section of the paper is divided into three smaller parts, where we study the following topics. First, we exhibit examples of when SSD points can be obtained in $\Free(M)$ for certain metric spaces $M$. Specifically, as we briefly mentioned earlier, we prove that for any metric space $M$, an element of the form $\mu = \sum_n a_n m_{x_n, x_{n+1}}$ with $\sum_n a_n = 1$ is an SSD point in $\Free(\cup_n [x_n, x_{n+1}])$ if and only if it is Fréchet differentiable in this space. 
Next, we observe that a Lipschitz-free space $\Free(M)$ is SSD if and only if $\Free(M)$ is finite-dimensional (cf. Proposition \ref{inf-dim}), so infinite dimensional Lipschitz-free spaces are forced to have points failing the SSD property. Then, in Subsection 3.2, we define locally uniformly non-aligned metric spaces. This enables us to prove that every molecule is an SSD point in $\Free(M)$, provided $M$ is a metric space of this kind. 

Finally, in the last subsection, we show that for any metric space $(M,d)$ and any $\gamma > 0$, it is possible to construct an associated locally uniformly non-aligned metric space $(M,d_\gamma)$, by defining the metric $d_{\gamma}$ over the same set $M$ as 
\[
d_\gamma (x,y) := 
\begin{cases}
d(x,y) + \gamma & \mbox{ if } x, y \in M \text{ with } x \neq y; \\ 
0 & \mbox{ if } x = y. 
\end{cases} 
\]
Let us mention that the metric $d_\gamma$ is not equivalent to the metric $d$, except in specific cases, such as when $(M,d)$ is a uniformly discrete metric space (i.e., $\inf\{d(x,y) : x\neq y \in M\} > 0$). Moreover, the structure of $\mathcal{F}(M, d_\gamma)$ may differ significantly from that of $\mathcal{F}(M,d)$. For instance, $\mathcal{F}(M,d_\gamma)$ has the Schur property and the Radon-Nikod\'ym property since $(M,d_\gamma)$ is uniformly discrete, which may not hold for $\mathcal{F}(M,d)$ particularly when $(M,d)$ contains a line segment (see \cite{AGPP, Kalton04}).

We conclude the paper by presenting our main result (cf. Theorem \ref{main1}), showing that every finitely supported element in $\Free(M, d_{\gamma})$ is an SSD point---in particular, deducing that the set of all SSD points in $\Free (M, d_\gamma)$ is dense, see Corollary \ref{main}. As a concrete application of this, any Lipschitz-free over a uniformly discrete metric is isomorphic to another Lipschitz-free space whose SSD points are a dense set---see Corollary \ref{cor:biLipschitz}.

\section{Preliminaries}

This section is dedicated to introducing the main tools we will be using throughout the paper. We start by presenting the definitions and main properties we will be using about Lipschitz-free spaces, and then we proceed to discuss SSD points, which form the central concept of the paper. Throughout the paper, we consider only {\it real} Banach spaces and denote by $B_X$ and $S_X$ the unit ball and unit sphere of a Banach space $X$, respectively.


\subsection{Lipschitz-free spaces} 
A \textit{pointed metric space} is just a metric space $M$ in which we distinguish an element, called 0.
Given a pointed metric space $M$ with distinguished point $0$, we write $\Lip(M)$ to denote the Banach space of all Lipschitz functions $f: M \longrightarrow \R$ which vanish at the distinguished point $0$. This space is endowed with the Lip-norm 
\begin{equation*}
    \|f\|_{\text{\rm Lip}} := \sup \left\{ \frac{|f(x) - f(y)|}{d(x,y)}: x, y \in M, x \not= y \right\}.
\end{equation*}
The canonical isometric embedding of $M$ into $\Lip(M)^*$ is denoted by $\dt$ and defined by $\langle f, \dt(x) \rangle = f(x)$ for all $x \in M$ and all $f \in \Lip(M)$. We denote by $\Free(M)$ the norm-closed linear space of $\dt(M)$ in the dual space $\Lip(M)^*$, which is usually called the \textit{Lipschitz-free space over $M$.}

A \textit{molecule} in $\Free(M)$ (for some authors, an \textit{elementary molecule}) is an element of the form 
\begin{equation*}
m_{x,y} := \frac{\dt (x) - \dt (y)}{d(x,y)}
\end{equation*}
where $x, y \in M$ are such that $x\not=y$. When $\mu$ is a finite sum of molecules, we have that it has an optimal representation \cite[Proposition 3.16]{Weaver}. In other words, we can write $\mu$ as follows
\begin{equation*}
    \mu = \sum_{i=1}^{n} \lambda_i m_{x_i, y_i}, \ \ \ \mbox{where} \ \ \  \lambda_i > 0 \ \ \ \mbox{with} \ \ \ \|\mu\| = \sum_{i=1}^n \lambda_i .
\end{equation*}
 It is easy to check and well known that for every $\mu \in S_{\Free (M)}$ and $\eps >0$, there exist $x_n \neq y_n \in M$ and $(a_n) \in \ell_1$ with $a_n>0$ such that 
\begin{equation*}\label{eq:mu}
\mu = \sum_{n} a_n m_{x_n,y_n} 
\end{equation*}
and $\sum_{n} a_n < 1 + \eps$. We send the reader to the survey \cite{Godefroy1} and the monograph \cite{Weaver} for a complete background.


Finally, recall that a metric space $M$ is called \textit{uniformly discrete} when $\inf \{ d(x, y) : x \neq y \in M \} > 0$.

\subsection{Strong subdifferentiable points} The norm $\| \cdot \|$ of a Banach space $X$ is \textit{strongly subdifferentiable} (SSD, for short) at $x \in S_X$ when the one-sided limit 
\begin{equation*} 
 \lim_{t \rightarrow 0^+} \frac{\|x + th\| - 1}{t}
\end{equation*}
exists uniformly in $h \in B_X$. For simplicity, we shall say that $x \in S_X$ is an SSD point if the norm of $X$ is SSD at $x$. When every $x\in S_X$ is an SSD point, $X$ will be said to be SSD.



From the very definitions, one can realize immediately that a norm of a Banach space $X$ is Fréchet differentiable at a point $x \in S_X$ if and only if it is both Gâteaux and SSD at this point. Also, it is well known in the field that \v Smulyan's Lemma provides a geometric interpretation of G\^ateaux and Fréchet differentiability of the norm. That is, a point $x\in S_X$ is a G\^ateaux differentiability point (resp., Fréchet differentiability point) if and only if the set $D(x):= \{ x^* \in S_{X^*}: x^*(x) = 1 \}$ is a singleton (resp., $D(x)\in S_{X^*}$ being a strongly exposed point by $x$)---see, for instance, \cite[Section 7.2]{FHHMZ}. In the same spirit, Franchetti and Payá in \cite{FP} provided a type of \v Smulyan lemma holds for SSD points by showing a point $x \in S_X$ is an SSD point if and only if $x$ strongly exposes the set $D(x):= \{ x^* \in S_{X^*}: x^*(x) = 1 \}$, i.e., given $\eps >0$, there exists $\eta(\eps,x)>0$ such that whenever $x^* (x) > 1-\eta(\eps,x)$ for some $x^* \in B_{X^*}$, we have that $\dist(x^*, D(x)) < \eps$ (see \cite[Theorem 1.2]{FP}). We will be using this geometric characterization for SSD points without any explicit reference throughout.

\section{Main results}

In this section, we start by presenting some introductory results as examples and remarks on $\Free(M)$ for some metric spaces $M$. 


\subsection{Some examples and remarks} Recall that, for a Banach space $X$, a point $x \in S_X$ is a point of Gâteaux differentiability if and only if there exists a unique $f \in S_{X^*}$ such that $f(x) = 1$. Let us recall also that given two points $x,y \in M$, we define the metric segment between them as the set 
\begin{equation*}
    [x,y] := \{ z \in M: d(x,z) + d(y,z) = d(x,y) \}. 
\end{equation*}

The next result states that an element, which is roughly speaking a convex series of ``consecutive'' molecules (that is, a convex sum of molecules are specifically of the form $m_{x_n, x_{n+1}}$) is a point of Gâteaux differentiability in the Lipschitz-free over the union of segments determined by the representation of the given element. Although the proof is straightforward, we include it for completeness. In what follows, the subsets of a given metric space are equipped with the restriction of the metric of that space.


\begin{proposition} \label{small-lemma} Let $M$ be a pointed metric space and $I$ be an index set. If $\mu = \sum_{n \in I} a_n m_{x_n,x_{n+1}} \in S_{\Free (M)}$ for some $\{x_n : n \in I\} \subseteq M$ containing $0$ and $a_n>0$ with $\sum_{n \in I} a_n=1$, then $\mu$ is a point of Gâteaux differentiability in $\Free(\cup_{n \in I} [x_n,x_{n+1}])$.
\end{proposition}

\begin{proof}  If $f \in S_{\Lip (M)}$ satisfies that $\langle f, \mu \rangle =1$, then by using a convexity argument, we have that $\langle f, m_{x_n,x_{n+1}} \rangle =1$ for every $n \in I$. For simplicity, say $I = \{n_1, n_2, \ldots\}$ and $x_{n_1} = 0$. As $f$ attains its Lipschitz constant between $x_{n_1}=0$ and $x_{n_2}$, the value of $f$ on the segment $[x_{n_1}, x_{n_2}]$ is determined, i.e., $f(t)= d(x_{n_1},t)$ for $t \in [x_{n_1}, x_{n_2}]$. From the fixed value of $f(x_{n_2})=d(x_{n_1},x_{n_2})$, we observe that the value of $f$ is also determined on $[x_{n_2}, x_{n_3}]$. In this way, the value of $f$ is uniquely determined on the entire segment $[x_n,x_{n+1}]$ for every $n \in I$, which proves that the finite sum of molecules $\mu$ must be a point of Gâteaux differentiability in $\Free(\cup_{n \in I} [x_n,x_{n+1}])$ as we wanted.
\end{proof}

\begin{proposition} \label{Prop:unifdiscrete} Let $M$ be a pointed metric space and $I$ be an index set. 
Let $\mu = \sum_{n \in I} a_n m_{x_n,x_{n+1}} \in S_{\Free (M)}$ for some $\{x_n : n \in I\} \subseteq M$ containing $0$ and $a_n>0$ with $\sum_{n \in I} a_n=1$. Consider the following statements.
\begin{enumerate}
    \itemsep0.25em
    \item $\mu$ is an SSD point in $\Free(\cup_{n \in I} [x_n,x_{n+1}])$.
    \item $\mu$ is a point of Fréchet differentiability in $\Free(\cup_{n \in I} [x_n,x_{n+1}])$.
    \item $\cup_{n \in I} [x_n,x_{n+1}]$ is uniformly discrete and bounded.
\end{enumerate} 
Then (1) $\Leftrightarrow$ (2) $\Rightarrow$ (3). \end{proposition}

\begin{proof} Notice from Proposition \ref{small-lemma} that the element $\mu$ is a point of Gâteaux differentiability in $\Free (\cup_{n \in I} [x_n,x_{n+1}])$. Thus, $\mu$ is an SSD point in $\Free (\cup_{n \in I} [x_n,x_{n+1}])$ if and only if $\mu$ is a point of Fréchet differentiability in $\Free (\cup_{n \in I} [x_n,x_{n+1}])$. This proves that (1) and (2) are equivalent. Now, suppose that $\mu$ is a Fréchet differentiability point in $\Free(\cup_{n \in I} [x_n, x_{n+1}])$. Then, the set $\cup_{n \in I} [x_n,x_{n+1}]$ must be uniformly discrete and bounded due to \cite[Theorem 2.4]{BLR}. 
\end{proof}

Implication (3) $\Rightarrow$ (2) in Proposition \ref{Prop:unifdiscrete} is {\it not} true in general.

\begin{example} \label{ex:ud-noF-branch}
Consider the metric space of the countably branching tree of height one, i.e., $M = \mathbb{N} \cup \{0\}$, where $d(0,n)=1$ for every $n \in \mathbb{N}$ and $d(n,m)=2$ for all $n \neq m \in \mathbb{N}$. The operator $\ell_1$ to $\Free(M)$ given by $\dt(e_n) := \dt(n)$ for every $n \in \mathbb{N}$ is a linear surjective isometry. Note that $\cup_{n \in \mathbb{N}} [n-1, n] = M$ is uniformly discrete and bounded while $\mu = \sum_{n=1}^\infty a_n m_{n-1, n}$ is not an SSD point in $\Free (M)$ since $\Free (M) = \ell_1$ does not have Fréchet differentiability points---see \cite{GMZ}. 
\end{example}

\begin{remark} Let us observe that Proposition \ref{small-lemma} does not apply to a general convex series of molecules. 
Indeed, consider the metric space $M = \{0,1,2,3\}$ in $\R$. Define the element $\mu:= \frac{1}{2} m_{1,0} + \frac{1}{2} m_{3,2}$. Then, the Lipschitz functions $f:=(0,1,2,3)$ and $g:=(0,1,0,1)$ both attain their norms at $\mu$. So, $\mu$ cannot be a Gâteaux differentiability point. We thank Ramón J. Aliaga for pointing this out.
\end{remark}

The next result implies that unless the cardinality of $M$ is finite, the set of SSD points in $\Free(M)$ should differ from the unit sphere, as this would imply that the space $\Free(M)$ is an Asplund space. 


\begin{proposition} \label{inf-dim} Let $M$ be a pointed metric space. The following are equivalent.

\begin{itemize}
\itemsep0.25em
\item[(a)] $\Free(M)$ admits a renorming which is SSD.
\item[(b)] Every renorming of $\Free(M)$ is SSD.
\item[(c)] $\Free(M)$ is finite dimensional.
\end{itemize}
\end{proposition}

\begin{proof} Let us prove that (a) implies (c). 
Note that a Lipschitz-free space $\Free(M)$
is an Asplund space only when it is finite-dimensional (that is, the cardinality of $M$ is finite) as any infinite-dimensional Lipschitz-free space contains an isomorphic copy of $\ell_1$ (see \cite{CDW, HN}). For (c) implying (b), notice that every finite-dimensional Banach space is SSD (as a consequence of Dini’s theorem).
\end{proof}

\subsection{Locally uniformly non-aligned metric spaces}

In this subsection, our aim is to introduce a metric space $M$ where \textit{every}  molecule exhibits the property of being an SSD point in $\Free(M)$. For any $x,y,z\in M$, the \textit{Gromov product of $x,y$ with respect to $z$} is defined as 
\begin{equation*} 
G_z(x,y):= d(x,z)+d(z,y)-d(x,y)
\end{equation*} 
(some authors consider the Gromov product as the above value divided by 2---see, for instance, \cite{BH, CCGMRZ}; in our context, this would not make much of a difference).

\begin{definition} Let $M$ be a pointed metric space and $x,y \in M$ with $x\neq y$. We say that the pair $(x,y)$ satisfies \textit{property (G)} if there exists $\eta:= \eta(x,y) > 0$ such that $G_z(x,y) > \eta$ for every $z \in M \setminus \{x,y\}$. The metric space $M$ is said to be \textit{locally uniformly non-aligned} if every pair $(x,y)$ of points in $M$ satisfies property (G). 
\end{definition}

If $(x,y)$ satisfies property (G), then it is clear that $[x,y] = \{x,y\}$. Let us mention that $m_{x,y}$ is an extreme point of $B_{\Free(M)}$ if and only if $G_z(x,y)>0$ for every $z\in M\backslash \{x,y\}$ (see \cite[Theorem 1.1]{AliPer20}). Notice also that a locally uniformly non-aligned metric space does not need to be uniformly discrete (for instance, take $M= \{(1, 0, \ldots, 0, \alpha_n, 0, 0, \ldots): n\geq 2\} \cup \{0\} \subseteq c_0$, where $(\alpha_n)$ is a sequence of positive real numbers converging to $0$).

\begin{example} \label{Petr-example} 
Let $M = \{x=0, y, z_1, z_2, \ldots \}$ be such that $d(x,y) = d(z_n, z_m) = 1$, $d(x,z_n) = \frac{1}{2}$ and $d(y,z_n) = \frac{1}{2} + \varepsilon_n$, where $(\varepsilon_n) \subseteq \mathbb{R}^+$ is such that $\varepsilon_n \rightarrow 0$ as $n \rightarrow \infty$ (see Figure \ref{fig:petrs}). It is clear that $G_{z_n} (x,y) = \eps_n \rightarrow 0$ as $ n \rightarrow \infty$; so $(x,y)$ fails to have property (G). 
\begin{figure}
    \centering
    \includegraphics[width=0.9\linewidth]{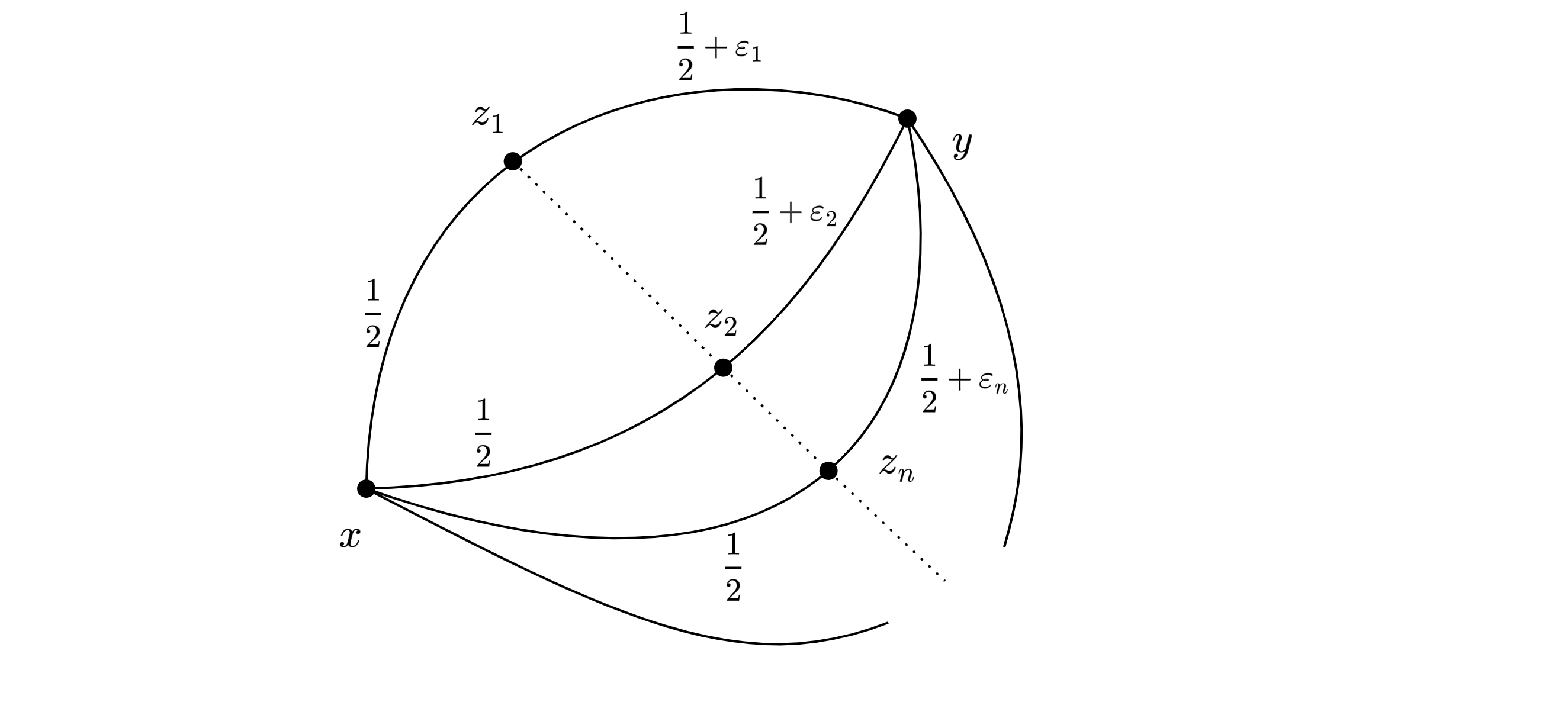}
    \caption{Example \ref{Petr-example}}
    \label{fig:petrs}
\end{figure}
\end{example}

To present a characterization for property (G) of a pair of points $(x,y)$, we introduce the following notion. 

\begin{definition}\label{def:peak}
    Let $M$ be a pointed metric space and $x\neq y \in M$. We say that $f \in S_{\Lip (M)}$ \textit{peaks at} $(x,y)$ if $f(m_{x,y}) =1$ and
    \begin{align*} \label{condition-f-xy}
    &\text{  $\sup \{|f(m_{p,q})|: (p,q) \neq (x,y) \} < \gamma <1$ for some $\gamma >0$.}   
\end{align*}
\end{definition}

Let us point out that our Definition \ref{def:peak} is different from the one introduced in \cite[Definition 5.2]{GPRZ}, while these two concepts coincide when the given points $x$ and $y$ are isolated points. For $x \neq y \in M$, recall also the auxiliary Lipschitz function $f_{x,y} \in S_{\Lip (M)}$ given by 
\begin{equation*} \label{fxy}
f_{x,y}(z):=\dfrac{d(x,y)}{2}  \dfrac{d(z,y)-d(z,x)}{d(z,y)+d(z,x)}
\end{equation*}
(see \cite[Lemma 3.6]{GPRZ} and \cite{IKW}).

\begin{proposition}\label{prop:characterization_G}
    Let $M$ be a pointed metric space and $x \neq y \in M$ be given. The following are equivalent. 
\begin{itemize}
\itemsep0.25em
\item[(1)] The pair $(x,y)$ satisfies property (G).
\item[(2)] The function $f_{x,y}$ peaks at $(x,y)$.
\item[(3)] There exists $f \in S_{\Lip(M)}$ which peaks at $(x,y)$.
\end{itemize}
\end{proposition}

\begin{proof}
(1) $\Leftrightarrow$ (2): Suppose that $f_{x,y}$ peaks at $(x,y)$ (witnessed by $0<\gamma<1$) and let $z \in M \setminus \{x,y\}$ be given. Notice first that $f_{x,y}(x) = \frac{d(x,y)}{2}$. Then, 
\begin{align*}
|f_{x,y}(m_{x,z})| = \frac{d(x,y)}{2 d(x,z)}  \left| 1 - \left( \frac{d(z,y) - d(z,x)}{d(z,y) + d(z,x)} \right) \right| = \frac{d(x,y)}{d(z,y) + d(z,x)} < \gamma.
\end{align*}
This implies that 
\begin{align*}
G_z(x,y) = d(x,z) + d(z,y) - d(x,y) > \frac{1}{\gamma} d(x,y) - d(x,y) = \left( \frac{1}{\gamma} - 1 \right) d(x,y).
\end{align*}
If we define $\eta := \left( \frac{1}{\gamma} - 1 \right)  d(x,y) > 0$, we are done. Conversely, let $\eta := \eta (x,y)$ be the one given by property (G) for the pair $(x,y)$. It is clear that $\langle f_{x,y},m_{x,y} \rangle=1$. Now, for any pair $(p,q)\neq (x,y)$, we get
 \begin{align*}
     \langle f_{x,y},m_{p,q} &\rangle \leq \dfrac{d(x,y)}{\max\{d(x,p)+d(p,y), d(x,q)+d(q,y)\}} \leq \dfrac{d(x,y)}{d(x,y)+\eta} < 1;
 \end{align*}
hence $f_{x,y}$ peaks at $(x,y)$.

It remains to show that (3) $\Rightarrow$ (1). Let $f \in S_{\Lip(M)}$ peak at $(x,y)$ with constant $0<\gamma<1$ and let $z \in M \setminus \{x,y\}$. Observe that
\begin{align*}
G_z(x,y) = d(x,z) + d(z,y) - d(x,y) &> \frac{1}{\gamma}(f(x) - f(z)) + \frac{1}{\gamma}(f(z) - f(y)) - d(x,y) \\
&= \frac{1}{\gamma}(f(x) - f(y)) - d(x,y) \\
&= \left( \frac{1}{\gamma} - 1 \right) d(x,y)
\end{align*}
for every $z \in M \setminus \{x,y\}$. This shows that $(x,y)$ satisfies property (G).
\end{proof}


Now we are ready to prove our main interest in introducing the aforementioned property.

\begin{theorem}\label{thm:SSD1}
    Let $M$ be a pointed metric space and $x, y \in M$ with $x \neq y$. If $(x,y)$ has property (G), then the molecule $m_{x,y}$ is an SSD point in $\Free(M)$. 
\end{theorem}

\begin{proof}
By Proposition \ref{prop:characterization_G}, there exists $f \in S_{\Lip(M)}$ such that $f$ peaks at $(x,y)$ with constant $0<\gamma<1$. 
Let $\varepsilon>0$ and take $g\in S_{\Lip(M)}$ to be such that $\langle g,m_{x,y}\rangle>1- \gamma_\varepsilon$, where $\gamma_\varepsilon > 0$ satisfies the following condition
\[
0<\gamma_\varepsilon<  \frac{\eps}{4-\eps} (1-\gamma).
\] 
 Put $h:= \left(1- \dfrac{\varepsilon}{4} \right) g+ \dfrac{\varepsilon}{4}f$ and observe that 
\[\langle h,m_{x,y} \rangle > \left(1-\dfrac{\varepsilon}{4} \right)(1-\gamma_\varepsilon)+\dfrac{\varepsilon}{4}.\]
By the choice of $\gamma_\eps$, we have that for $(p,q)\neq (x,y)$, 
\[|\langle h, m_{p,q} \rangle| \leq \left(1- \dfrac{\varepsilon}{4} \right) + \dfrac{\varepsilon}{4} \gamma< \langle h, m_{x,y} \rangle.\]
It follows that $\|h\|= \langle h,m_{x,y} \rangle > \left(1-\dfrac{\varepsilon}{4} \right)(1-\gamma_\varepsilon)+\dfrac{\varepsilon}{4}$. Finally, considering $\hat h:= \dfrac{h}{\|h\|}$, we get 
\[
    \|\hat h-g\|\leq (1-\|h\|)+ \dfrac{\varepsilon}{2} < 1- \left(1-\dfrac{\varepsilon}{4} \right)(1-\gamma_\varepsilon)+ \dfrac{\varepsilon}{4}.
\]
Now it is enough to take small enough $\varepsilon$ and $\gamma_\varepsilon$. 
\end{proof}

As a direct consequence of Proposition \ref{prop:characterization_G} and Theorem \ref{thm:SSD1}, we have the following corollary.

\begin{corollary} \label{theorem:LUNA} Let $M$ be a pointed metric space. If $M$ is locally uniformly non-aligned, then every molecule is an SSD point in $\Free(M)$.
\end{corollary}

As a matter of fact, the idea of the proof of Theorem \ref{thm:SSD1} will be important in a more refined way when we prove our main result, Theorem \ref{main1}.

\begin{remark}\label{rem:prop_G}
    The sufficient condition of having property (G) for a pair of points {does not} characterize SSD molecules. The examples below are provided to illustrate this fact.
    \begin{enumerate}
        \itemsep0.25em
        \item Let $M=\{x,y,0\}$ be a metric space of three points, with $0\in[x,y]$ and $x\not= y$. Then, $m_{x,y}$ (as every element of the space) is an SSD point, but clearly, the pair $(x,y)$ does not have the property (G), as $G_0(x,y)=0$.
        \item Let $M = \mathbb{N} \cup \{0\}$ with the usual metric inherited from $\mathbb{R}$. Then the linear operator from $\Free (M)$ to $\ell_1$ sending $\dt(n)$ to $\sum_{k=1}^n e_k$ for $n \in \mathbb{N}$ is a surjective isometry. In this case, every molecule $m_{x,y}$ with $x \neq y \in M$ is an SSD point in $\Free(M)$. However, it is easy to check that $m_{x,y}$ satisfies property (G) if and only if $|x-y|=1$. 
        \item Consider $\cantor$ to be the usual middle third Cantor set with the usual distance. The set $\cantor$ may be written as $\cantor=[0,1]\backslash \big( \bigcup_{n=1}^\infty I_n \big)$, where the sequence of intervals $I_n:=]a_n,b_n[$ is an enumeration of the middle thirds we are taking on each step( in particular, $I_1=]\frac{1}{3},\frac{2}{3}[$, $I_2= ]\frac{1}{9},\frac{2}{9}[$, $I_3=]\frac{7}{9}, \frac{8}{9}[$, etc.,). It is well known that $\Free (\cantor)$ is isometrically isomorphic to $\ell_1$. Observe that $m_{a_n, b_m}$ is an SSD point in $\Free(\cantor)$ if and only if $n=m$, while $m_{x,y}$ fails to have property (G) for any $x \neq y \in \cantor$. 
    \end{enumerate}
\end{remark}

\begin{remark} 
 Recall the metric space $M$ and the points $x \neq y \in M$ from Example \ref{Petr-example}. It was previously observed that $(x, y)$ fails property (G) (equivalently, thanks to Proposition \ref{prop:characterization_G}, there is no $f \in S_{\Lip(M)}$ that peaks at $(x, y)$). However, as we observed in Remark \ref{rem:prop_G}, the fact that $(x, y)$ lacks property (G) does not provide any information about whether $m_{x, y}$ is an SSD point in $\Free(M)$. 
 
We claim that, in this case, $m_{x,y}$ is indeed an SSD point in $\Free (M)$. To this end, let $\eps >0$ be given and let us find $n_0 \in \mathbb{N}$ such that $0< \eps_n < \eps$ for all $n > n_0$. We consider $M_0 := \{x=0, y, z_1,\ldots, z_{n_0}\}$. Take $0<\gamma<\eps$ sufficiently small so that 
whenever $g \in S_{\Lip (M_0)}$ satisfies that $g(m_{x,y})>1-\gamma$, there exists $h \in S_{\Lip (M_0)}$ so that $h(m_{x,y})=1$ and $\|h-g\| < \eps$. Now, suppose that $f \in S_{\Lip(M)}$ satisfies that $f(m_{x,y}) > 1-\gamma$. Then 
\begin{eqnarray*} 
1 = d(x,y) < \frac{f(x) - f(z_n) + f(z_n) - f(y)}{1-\gamma} 
&\leq& \frac{d(x,z_n) + d(z_n,y)}{1-\gamma} \\
&=& \frac{1+ \varepsilon_n}{1-\gamma} = 1 + \frac{\varepsilon_n + \gamma}{1 - \gamma}.
\end{eqnarray*}
Thus, we get that $d(z_n, y) - (f(z_n) - f(y))$ and $d(x,z_n)- ( f(x) - f(z_n))$ are smaller than or equal to $\varepsilon_n + \gamma$. This yields that 
 \begin{equation}
    f(m_{x,z_n}) \geq 1 - (2 \eps_n+2\gamma) \, \text{ and } \, f(m_{z_n,y}) \geq 1- \frac{2\eps_n+2\gamma}{1+2\eps_n}. \label{xzn2}
 \end{equation}
 Letting $\tilde{f} := \|f\vert_{M_0}\|^{-1} f \vert_{M_0}$, we have $\tilde{f}(m_{x,y})>1-\gamma$. Thus, there exists $h \in S_{\Lip(M_0)}$ such that $h(m_{x,y})=1$ and $\|h-\tilde{f}\| < \eps$. Extend $h$ to a norm one Lipschitz map on $M$ and denote it again by $h$. As before, we observe that 
   \begin{equation}
    h(m_{x,z_n}) \geq 1 - 2 \eps_n \, \text{ and } \, h(m_{z_n,y}) \geq 1- \frac{2\eps_n}{1+2\eps_n}. \label{xzn3}
 \end{equation}
 Finally, we estimate $\|h-f\|$ by dividing some cases. We use (\ref{xzn2}) and (\ref{xzn3}).
 \begin{enumerate}
     \itemsep0.25em
     \item $|(h-f)(m_{x,y})| < \gamma < \eps$.
     \item If $1 \leq n \leq n_0$, then 
     \[
     |(h-f)(m_{x,z_n})| \leq \| h \vert_{M_0} - \tilde{f} \| + |\|\tilde{f}\|-1| < \eps + \gamma < 2\eps. 
     \]
     Similarly, $|(h-f)(m_{z_n,y})| < 2 \eps$. 
     \item If $n > n_0$, then 
     \[
     |(h-f)(m_{x,z_n})| \leq 2\eps_n+2\gamma < 2 \eps + 2 \eps = 4 \eps.
     \]
     Similarly, $|(h-f)(m_{z_n,y})| \leq \frac{2\eps_n+2\gamma}{1+2\eps_n} < 4\eps$. 
     \item Let $i, j \in \mathbb{N}$. If $1\leq i,j \leq n_0$, then $|(h-f)(m_{z_i,z_j})| \leq \| h\vert_{M_0} - f \vert_{M_0}\| < 2 \eps$. Now suppose that $1\leq i \leq n_0$ and $j > n_0$. Let us recall that $d(z_i, z_j) =1$ and $d(x,z_i)=d(x,z_j)=1/2$. Using this and the previous cases, we get that
\begin{eqnarray*} 
     |(h-f)(m_{z_i,z_j})| &\leq& d(x,z_i)|(h-f)(m_{x,z_i})| + d(x,z_j) |(h-f)(m_{x,z_j})| \\
     &<& \frac{2\eps}{2}  + \frac{4\eps}{2}  = 3 \eps. 
\end{eqnarray*} 
     Similarly for $i, j > n_0$, we have $ |(h-f)(m_{z_i,z_j})| < 4 \eps$. 
 \end{enumerate}
Therefore, by using (1)--(4), we observe that $h \in S_{\Lip (M)}$ is such that $h(m_{x,y})=1$ and satisfies $\|h-f\| \leq 4\eps$. This shows that the molecule $m_{x,y}$ is an SSD point in $\Free (M)$ as we wanted to show.
\end{remark}


Let us now conclude this subsection by considering some relations between property (G), uniformly Gromov rotundness (see \cite{BH, CCGMRZ}), and uniformly concaveness. Let $M$ be a pointed metric space and $A$ a set of molecules in $M$. Recall that $A$ is \textit{uniformly Gromov rotund} if there exists $\delta = \delta(A) > 0$ such that 
\begin{equation*}
    G_z(x,y) > \delta \min \{d(x,z), d(y,z)\} 
\end{equation*}
whenever $m_{x,y} \in A$ and $z \in M \setminus \{x,y\}$. For simplicity, let us say that the pair $(x,y)$ is uniformly Gromov rotund if the singleton $\{m_{x,y}\}$ is uniformly Gromov rotund.

Also, recall that a metric space $M$ is called \textit{uniformly concave} if for every ${x \neq y \in M}$ and $\eps >0$ there exists $\delta >0$ such that 
$G_z (x,y) > \delta$ for every $z \in M$ satisfying $\min \{ d(x,z), d(y,z) \} \geq \eps$ (see \cite[Definition 3.3]{Weaver}). 
Let us say that a pair of points $(x,y)$ is uniformly concave if for any $\eps >0$, there exists $\delta = \delta(x,y,\eps)>0$ such that $G_z (x,y) > \delta$ for every $z \in M$ satisfying $\min \{ d(x,z), d(y,z) \} \geq \eps$.

As the definitions suggest, property (G), uniformly Gromov rotundness, and uniformly concaveness are interrelated. One key distinction between them is that the value $\eta=\eta(x,y)$ in the definition of property (G) for $(x,y)$ does not depend on the choice of a point $z \in M \setminus \{x,y\}$, whereas the lower bound for $G_z(x,y)$ in the definition of uniformly Gromov rotundness does depend on the choice of $z \in M \setminus \{x,y\}$. Additionally, in the definition of uniform concaveness of $(x,y)$, unlike the other two definitions, not all points $z \in M \setminus \{x,y\}$ are considered; only those $z \in M$ that satisfy a certain condition related to $\min \{d(x,z), d(y,z)\}$.



\begin{proposition} \label{propG2} Let $M$ be a pointed metric space. Let $x, y \in M$ be such that $x \not= y$. 
\begin{enumerate}
    \itemsep0.25em 
    \item If the pair $(x,y)$ satisfies property (G), then it is uniformly Gromov rotund. 
        \item If the pair $(x,y)$ is uniformly Gromov rotund, then it is uniformly concave. 
    \item If $x$ and $y$ are isolated points in $M$, then $(x,y)$ satisfies property (G) if and only if $(x,y)$ is uniformly Gromov rotund if and only if $(x,y)$ is uniformly concave.     
\end{enumerate}
\end{proposition}

\begin{proof} 
(1): Suppose that $M$ is any metric space and that $(x,y)$ satisfies property (G). Then, there exists $\eta = \eta(x,y) > 0$ such that $G_z(x,y) > \eta$ for every $z \in M \setminus \{x,y\}$. We have two cases. If $d(z,y) > d(x,y)$ or $d(x,z) > d(x,y)$, then we have that $G_z(x,y) > d(x,z)$ or $G_z(x,y) > d(y,z)$. In any case, we have that $G_z(x,y) > \min \{d(x,z), d(y,z)\}$. On the other hand, if $d(z,y)$ and $d(x,z)$ are both $\leq$ $d(x,y)$, then we have that 
\begin{align*}
G_z(x,y) > \eta &\geq \frac{\eta}{d(x,y)} \max \{ d(x,z), d(z,y) \}. 
\end{align*}
Putting together these cases, we get that 
\begin{equation*}
    G_z(x,y) > \min \left\{ 1, \frac{\eta}{d(x,y)} \right\} \min \{d(x,z), d(z,y)\} 
\end{equation*}
for every $z \in M \setminus \{x,y\}$. This shows that $\delta= \delta(\{x,y\}) := \min \left\{ 1, \frac{\eta}{d(x,y)} \right\}$ does the job.

(2): This is obvious by definition. 


(3): Suppose now that $x,y\in M$ are isolated and $(x,y)$ is uniformly concave. Take $r_x, r_y >0$ so that $B(x, r_x) = \{x\}$ and $B(y, r_y) = \{y\}$, respectively. Considering $\eps(x,y) := \min \{r_x, r_y\}$, we can find $\delta = \delta(x,y,\eps(x,y))>0$ such that $G_z (x,y) > \delta$ for every $z \in M \setminus \{x,y\}$ \, (noting that $\min \{d(x,z), d(y,z)\} \geq \min\{r_x,r_y\}$). This proves that $(x,y)$ has property (G).
\end{proof}

\begin{remark} 
\begin{enumerate}
        \itemsep0.25em
    \item The converse of (1) in Proposition \ref{propG2} is {\it not} true in general. Indeed, consider $M = \{x=0, y, z_1, z_2, \ldots \}$ endowed with the metric given by $d(x,y) = 1$, $d(x,z_n) = \frac{1}{2n}$, $d(y,z_n) = 1 - \frac{1}{4n}$ and $d(z_n, z_m) = \frac{1}{2n}+\frac{1}{2m}$ for every $n,m \in \N$ with $n\not=m$. On the one hand, we have that 
\begin{equation*}
    G_{z_n}(x,y) = \frac{1}{2n} + 1 - \frac{1}{4n} - 1 > \eps \, \min \{ d(x,z_n), d(y,z_n)\} 
\end{equation*}
for every $n \in \N$, whenever $0<\eps<\frac{1}{2}$. This shows that the pair $(x,y)$ is uniformly Gromov rotund, while $G_{z_n}(x,y)\rightarrow 0$ as $n\rightarrow\infty$; so $(x,y)$ fails property (G).
\item The converse of (2) does not generally hold as well. Recall that for $x \neq y \in M$, the pair $(x,y)$ is uniformly Gromov rotund if and only if the molecule $m_{x,y}$ is a strongly exposed point of $B_{\mathcal{F}(M)}$ \cite[Proposition 1.1 (b)]{CCGMRZ}. 
On the other hand, $(x,y)$ is uniformly concave if and only if $m_{x,y}$ is a preserved extreme point of $B_{\mathcal{F}(M)}$ \cite[Theorem 4.1]{AG} (which is equivalent to $m_{x,y}$ being a denting point of $B_{\mathcal{F}(M)}$ thanks to \cite[Theorem 2.4]{GPPR}). Finally, \cite[Example 6.4]{GPPR} shows that there is a (compact) metric space $M$ for which there exists a molecule $m_{x,y}$ which is a preserved extreme point but not a strongly exposed point of $B_{\mathcal{F}(M)}$. 
\end{enumerate}
\end{remark}

\subsection{The subset of SSD points is dense}
 

Let $M$ be a pointed metric space with metric $d$. One natural way of producing a locally uniformly non-aligned metric space from $M$ is to consider the metric space $(M,d_\gamma)$ with $\gamma >0$, where $d_\gamma : M \times M \rightarrow \R$ is given by $d_\gamma (x,y) := d(x,y) + \gamma$ for every $x, y \in M$ with $x \not=y$ and $d_\gamma (x,x) := 0$ for every $x \in M$. It is obvious from Corollary \ref{theorem:LUNA} that every molecule in $\Free (M, d_\gamma)$ is an SSD point. We have the following result, which will lead to the conclusion that every finite sum of molecules in $\Free (M, d_\gamma)$ is indeed an SSD point.

\begin{theorem} \label{main1} Let $(M,d)$ be a pointed metric space and $\gamma >0$. Then every element of 
\[
\left\{ \sum_{i=1}^n \lambda_i \left( \frac{\delta(x_i) - \delta(y_i) }{d(x_i, y_i) + \gamma} \right) : \exists f \in S_{\Lip(M,d)} : f(m_{x_i, y_i}) = 1,  \lambda_i >0, \sum_{i=1}^n \lambda_i = 1 \text{ and } n \in \mathbb{N}\right\} 
\]
is an SSD point in $\Free (M, d_{\gamma})$. 
\end{theorem}






\begin{proof} 
Let $\mu \in \Free (M, d_{\gamma})$ be an element of the following form 
\begin{equation*}
    \mu = \sum_{i=1}^{n} \lambda_i \left( \frac{\delta(x_i)-\delta(y_i)}{d(x_i,y_i)+\gamma} \right) 
\end{equation*}
where there exists $f \in S_{\Lip (M,d)}$ such that $f(x_i)-f(y_i)=d(x_i,y_i)$ for every $i=1,\ldots, n$ and $\sum_{i=1}^n \lambda_i = 1$ with $\lambda_i>0$. Let us notice that the sets $\{x_i\}_{i=1}^n$ and $\{y_i\}_{i=1}^n$ are disjoint. Indeed, let $g \in S_{\Lip (M, d_\gamma) }$ be such that $g(x_i)-g(y_i) = d(x_i,y_i) + \gamma$ for every $i = 1,\ldots, n$. Assuming for simplicity that $x_1 = y_2$, we have 
\begin{eqnarray*} 
g(x_2) = g(y_2) + d(x_2,y_2) + \gamma &=& g(y_1) + d(x_1,y_1) + d(x_2,x_1) + 2 \gamma \\
&\geq& g(y_1) + d(y_1, x_2) + 2\gamma.
\end{eqnarray*}
This yields a contradiction since 
\[
1 \geq \frac{g(x_2)-g(y_1)}{d(y_1,x_2)+\gamma} \geq \frac{d(y_1, x_2) + 2\gamma}{d(y_1, x_2)+\gamma} > 1.
\]
Setting $N:= \{0, x_1, y_1, \ldots, x_n, y_n \} \subseteq M$, it is clear that $f \vert_{N} \in S_{\Lip (N,d)}$ satisfies that $|f\vert_N| \leq \beta := \max \{ d(0, q) : q \in N \}$. Applying McShane extension theorem, we can extend $f \vert_N$ to $(M,d)$ whose values are contained in $[-\beta,\beta]$ without increasing its norm. We still denote this extension by $f$.

 \vspace{0.2cm}

Now let us define $f_\gamma \in \Lip (M, d_\gamma)$ as 
\begin{equation*} \label{modification1}
f_\gamma (z) :=
\begin{cases}
f(x_i) + \gamma, & \mbox{if} \ z = x_i, \, \forall i = 1,\ldots, n, \\
f(y_i), & \mbox{if} \ z = y_i, \, \forall i = 1,\ldots, n, \\
f(z) + \frac{\gamma}{2}, &  \mbox{if} \ z \not\in N.
\end{cases}
\end{equation*}
Observe that $\| f_\gamma \|_{\Lip(M,d_\gamma)} = 1$. As a matter of fact, for $1 \leq i, j \leq n$ and $p \not\in N$ we have 
\begin{enumerate}
    \itemsep0.25em
    \item $|f_\gamma (x_i) - f_\gamma(y_j)| = |f(x_i)+ \gamma - f(y_j)| \leq d(x_i, y_j)+\gamma$; 
    \item $|f_\gamma (x_i) - f_\gamma (p)| = | f(x_i) + \gamma - (f(p)+\frac{\gamma}{2}) |  \leq d(x_i,p) + \frac{\gamma}{2}$;  
    \item $|f_\gamma (y_j)-f_\gamma(p)| = |f(y_i) - (f(p)+\frac{\gamma}{2}) | \leq d(y_i, p) + \frac{\gamma}{2}$.  
\end{enumerate}
Thus, we have 
\begin{equation*}
    \langle f_\gamma, \mu \rangle = \sum_{i=1}^n \frac{\lambda_i}{d(x_i, y_i) + \gamma} (f(x_i) + \gamma - f(y_i)) = \sum_{i=1}^n \lambda_i = 1.
\end{equation*}
As it is clear that $\|\mu \|_{\Free (M,d_\gamma)} \leq 1$, we actually obtain that $\|\mu\|_{\Free (M,d_\gamma)} =1$.  
Let us note that 
\begin{equation}\label{eq:f_gamma}
    |f_\gamma(p)| \leq \beta + \frac{\gamma}{2} \, \text{ for every $p \not\in N$. }
\end{equation} 

\vspace{0.2cm} 
We also consider $\xi : \mathbb{R}_{\geq 0}\rightarrow [0,1]$ to be defined as 
    $\xi(t) = 1$ for $0\leq t \leq \beta$,
    \[
    \xi(t) = -\frac{1}{T- \beta } (t- T) \quad \text{for } \beta < t \leq T
    \]
    and $\xi(t) = 0$ for all $t \geq T$, where $T>0$ is sufficiently large (which will be specified later). Now, we define $G_\gamma \in \Lip(M,d_\gamma)$ given by 
    \begin{equation*} 
    G_\gamma (z) := f_\gamma(z) \, \xi ( d(0, z)  ) 
    \end{equation*} 
    for every $z \in M$. By using the definitions of $\xi$ and $f_\gamma$, we have immediately that $\langle G_\gamma, \mu \rangle=1$. For the remainder of the proof, and by abuse of notation, we shall denote by $m_{p,q}$ the elementary molecule $\frac{\delta(p) - \delta(q)}{d(p,q) + \gamma}$ in $\mathcal{F}(M, d_\gamma)$ for $p \neq q \in M$.

\vspace{0.25em}
    \textbf{Claim A}: $\sup \{ |\langle G_\gamma, m_{p,q} \rangle| : p \not\in N \text{ or } q \not\in N\} < 1$. 
    \vspace{0.25em}
    
In order to prove the claim, we consider different cases. Let $p \not\in N$ or $q \not\in N$ with $p \not= q$. If $d(0,p), d(0,q) \leq \beta$, then 
    \begin{align*}
        |\langle G_\gamma, m_{p,q} \rangle| &= \frac{|f_\gamma(p) \xi(d (0,p)) - f_\gamma(q) \xi(d (0,q))|}{d(p,q)+\gamma} \\
        &=  \frac{|f_\gamma(p)  - f_\gamma(q) |}{d(p,q)+\gamma} \\
        &\markthis{(a)}{b:1}{\leq}   \frac{d(p,q) + \frac{\gamma}{2}}{d(p,q) + \gamma} \markthis{(b)}{b:2}{\leq} \frac{2 \beta + \frac{\gamma}{2}}{2 \beta + \gamma} < 1,
   \end{align*}
   where the two inequalities \ref{b:1} and \ref{b:2} follow from the definition of $f_\gamma$ and the fact $d(p,q) \leq 2\beta$, respectively. 
Otherwise, we may (and we do) assume without loss of generality that $d(0,q)>\beta$ (in particular, $q \not\in N$). In this case, we have that
    \begin{align*}
        | \langle G_\gamma, m_{p,q} \rangle | &\leq \xi (d (0,p))  \frac{|f_\gamma(p)-f_\gamma(q)|}{d(p,q)+\gamma} + |f_\gamma(q)| \frac{|\xi (d(0,p))- \xi (d(0,q))|}{d(p,q) +\gamma} \\
        &\leq \frac{|f_\gamma(p)-f(q)-\frac{\gamma}{2}|}{d(p,q)+\gamma} + |f_\gamma(q)| \frac{|\xi (d(0,p))-\xi (d (0,q))|}{|d(0,p)-d(0,q)|} \frac{|d(0,p)-d(0,q)|}{d(p,q) +\gamma} \\
        &\stackrel{(\ref{eq:f_gamma})}{\leq} \frac{|f(p)-f(q)| +\frac{\gamma}{2}}{d(p,q)+\gamma}  + \left( \beta+ \frac{\gamma}{2} \right)  \left(  \frac{1}{T-\beta}\right)  \frac{d(p,q)}{d(p,q) +\gamma} =: (I)
    \end{align*}
We shall see that (I) is strictly smaller than $1$ uniformly for $p$ and $q$, provided that $T$ is sufficiently large. 
For simplicity, let us set  
    \[
    T_0 : = \frac{\gamma}{4} \left( \frac{T- \beta}{ \beta +\frac{\gamma}{2}} \right).
    \]
    At this moment, say $T>0$ is large enough so that 
\begin{equation}\label{bigT}
    \frac{2\beta+\frac{3\gamma}{2}}{T_0+\gamma} + \left( \beta+\frac{\gamma}{2} \right) \left( \frac{1}{T-\beta} \right) < \frac{3}{4}.
\end{equation}
We consider two cases. Suppose first that $d(p,q)<T_0$, which is equivalent to
\begin{equation}\label{eq:T_beta_delta}
\frac{\beta+\frac{\gamma}{2}}{T-\beta} < \frac{\gamma}{4} \frac{1}{d(p,q)}.
\end{equation}
By \eqref{eq:T_beta_delta}, we get that 
\begin{align*}
(I) &\leq \frac{d(p,q)+\frac{\gamma}{2}}{d(p,q)+\gamma} + \left(\beta+\frac{\gamma}{2}\right) \left(  \frac{1}{T-\beta}\right)  \frac{d(p,q)}{d(p,q) +\gamma} < \frac{d(p,q)+\frac{3\gamma}{4} }{d(p,q)+\gamma} < \frac{T_0 + \frac{3\gamma}{4}}{T_0 + \gamma} < 1.
\end{align*} 
On the other hand, if $d(p,q) \geq T_0$, then by \eqref{bigT} we get that 
\begin{align*}
    (I) 
    &\leq \frac{2\beta+\frac{3\gamma}{2}}{T_0 + \gamma} + \left(\beta+\frac{\gamma}{2} \right) \left( \frac{1}{T-\beta} \right) < \frac{3}{4}. 
\end{align*}
This completes the proof of the claim. Thus, we can find $K>0$ sufficiently large such that the following holds true
\[
 \sup \{ |\langle G_\gamma, m_{p,q} \rangle |: p \not\in N \text{ or } q \not\in N \} < \frac{K+\frac{\gamma}{2}}{K+\gamma} < 1. 
\]

Let us now move towards the end of the proof. Let $\varepsilon > 0$ be given and $g \in S_{\Lip(M, d_\gamma)}$ be such that $\langle g, \mu \rangle > 1 - \rho$, where $\rho = \rho(\varepsilon, \mu) > 0$ is small enough such that\footnote{Let us notice that (\ref{parameter1}) is indeed possible to be considered since 
\begin{equation*}
(1 - \sqrt{\varepsilon}) + \sqrt{\varepsilon} \left( \frac{ K + \frac{\gamma}{2}}{K + \gamma} \right) +  \frac {\beta \varepsilon}{\gamma}
= 1 - \sqrt{\varepsilon} \left(1 - \frac{K+ \frac{\gamma}{2}}{K + \gamma} - \frac{\beta \sqrt{\varepsilon}}{\gamma} \right) 
\end{equation*}
is smaller than $1$ for a sufficiently small $\varepsilon >0$, and then we can take $\rho=\rho(\varepsilon,\mu) > 0$ small enough in order to satisfy (\ref{parameter1}).}
\begin{equation} \label{parameter1}
(1 - \sqrt{\varepsilon}) + \sqrt{\varepsilon} \left( \frac{K + \frac{\gamma}{2}}{K + \gamma} \right) +  \frac {\beta \varepsilon}{\gamma} < (1 - \sqrt{\varepsilon})(1 - \rho) + \sqrt{\varepsilon}.
\end{equation}
\vspace{0.25em}

\textbf{Claim B}: $\exists\, \psi \in S_{\Lip (M, d_\gamma)}$ such that $\langle \psi, \mu \rangle = 1$ and $\| \psi - g\| \approx 0$.
\vspace{0.25em}

First, define $h: (M, d_\gamma) \rightarrow \R$ by 
\begin{equation*} \label{function1}
    h:= (1 - \sqrt{\varepsilon})g + \sqrt{\varepsilon} \,G_\gamma.
\end{equation*}
Then, we have that $\langle h, \mu \rangle > (1 - \sqrt{\varepsilon})(1 - \rho) + \sqrt{\varepsilon}$ and 
\begin{equation*}
    |\langle h, m_{p,q} \rangle| \leq (1 - \sqrt{\varepsilon}) + \sqrt{\varepsilon} \left( \frac{K + \frac{\gamma}{2}}{K + \gamma}\right)
\end{equation*}
for every $p\not\in N$ or $q \not\in N$. This implies that 
\begin{eqnarray*}
\langle h, \mu \rangle &>& (1 - \sqrt{\varepsilon})(1 - \rho) + \sqrt{\varepsilon} \\
&\stackrel{(\ref{parameter1})}{>}& (1 - \sqrt{\varepsilon}) + \sqrt{\varepsilon} \left( \frac{K + \frac{\gamma}{2}}{K + \gamma} \right) +  \frac {\beta \varepsilon}{\gamma} \\
&>&  \sup \{ | \langle h, m_{p,q} \rangle| : p \not\in N \text{ or } q \not\in N \}.
\end{eqnarray*}
In particular, this shows that $\|h\| = \| h \vert_N \|$ with $\langle h|_N, \mu \rangle > (1 - \sqrt{\varepsilon})(1 - \rho) + \sqrt{\varepsilon}$.  Consider $\mu$ as an element of the finite-dimensional space $\Free(N, d_\gamma)$, and recall that the norm of a finite-dimensional normed space is SSD. Thus, by letting $\rho=\rho(\eps,\mu)$ even smaller if necessary, we can find a Lipschitz function $\varphi \in \Lip(N, d_\gamma)$ such that $\|\varphi\| = \langle \varphi, \mu \rangle = \|h\|$ and $\|\varphi - h|_N\|_{\Lip(N, d_\gamma)} < \varepsilon$. Now, we define $\psi: (M, d_\gamma) \rightarrow \R$ by 
\begin{equation*}
\psi(p) :=
\begin{cases}
\varphi(p), & \forall p \in N, \\
h(p), & \forall p \in M \setminus N
\end{cases}
\end{equation*}
which will serve as the required element for our claim.

Notice that $\langle \psi, \mu \rangle = \langle \varphi, \mu \rangle = \|\varphi\|$. Moreover, $\|\psi - h\|_{\Lip(M, d_\gamma)}$ can be estimated as follows. If $p, q \in N$, then 
\begin{equation*}
    \frac{|(\psi - h)(p) - (\psi-h)(q)|}{d(p,q)+\gamma} = \frac{|(\varphi - h|_N)(p) - (\varphi-h|_N)(q)|}{d(p,q)+\gamma} < \varepsilon.
\end{equation*}
 On the other hand, if $p \in M \setminus N$ and $q \in N$, then 
 \begin{align*}
    \frac{|(\psi - h)(p) - (\psi-h)(q)|}{d(p,q)+\gamma} &= \frac{|(h-h)(p) - (\varphi - h)(q)|}{d(p,q)+\gamma} \\
    &\leq   \| \varphi - h \vert_N \|_{ \Lip(N, d_\gamma) } \, \frac{ \beta }{ \gamma} < \frac{\beta \eps }{\gamma}. 
 \end{align*}
Finally, the case when $p,q \in M \setminus N$ is trivial. Consequently, we get that 
\begin{equation*}
    \| \psi - g\|_{\Lip(M, d_\gamma)} \leq \| \psi - h \|_{\Lip(M, d_\gamma)} + \|h-g\|_{\Lip(M, d_\gamma)} < \max\left\{ \varepsilon, \frac{\beta\eps}{\gamma}\right\} + 2\sqrt{\varepsilon}. 
\end{equation*}

\vspace{0.2cm}

It remains to prove that $\psi$ attains its norm at $\mu$. To this end, observe that if $p,q \in M \setminus N$, then 
\begin{equation*}
    |\langle \psi, m_{p,q}\rangle| = | \langle h, m_{p,q} \rangle| \leq \|h\| = \|\varphi\| =\langle \psi, \mu \rangle. 
\end{equation*}
If $p,q \in N$, then 
\begin{equation*}
| \langle \psi, m_{p,q} \rangle| = | \langle \varphi, m_{p,q} \rangle| \leq \|\varphi\| = \langle \psi, \mu \rangle. 
\end{equation*} 
Finally, if $p \in N$ and $q \in M \setminus N$, then 
\begin{eqnarray*}
 | \langle \psi, m_{p,q} \rangle| = \frac{| \psi(p) - \psi(q) | }{d(p,q) + \gamma} &\leq& \frac{|\varphi(p) - h(p)| + |h(p) - h(q)|}{d(p,q) + \gamma} \\
 &\leq& \frac{ \|\varphi - h|_N \|_{\Lip(N, d_\gamma)} \, \beta}{\gamma}  + |\langle h, m_{p,q} \rangle| \\
 &<& \frac {\beta \varepsilon}{\gamma} + (1 - \sqrt{\epsilon}) + \sqrt{\varepsilon} \left( \frac{K + \frac{\gamma}{2}}{K + \gamma} \right) \\
 &\stackrel{(\ref{parameter1})}{<}& (1 - \sqrt{\varepsilon})(1 - \rho) + \sqrt{\varepsilon} \\
 &<& \|h\| = \|\varphi\| = \langle \psi, \mu \rangle.
\end{eqnarray*}
This completes the proof of the claim. Therefore, we conclude that $\mu$ is an SSD point in $\Free(M, d_\gamma)$ as we wanted to prove.
\end{proof}

Now we are ready to prove the denseness of the SSD points in $\Free (M, d_\gamma)$ for every $\gamma>0$. This is a consequence of Theorem \ref{main1}.

\begin{corollary} \label{main}
    Let $(M,d)$ be a pointed metric space and $\gamma >0$. Then every finite sum of molecules in $\Free (M, d_\gamma)$ is an SSD point. In particular,
 the set of SSD points in $\Free (M, d_\gamma)$ is dense. 
\end{corollary}

\begin{proof} Let $\gamma > 0$ be arbitrary. For simplicity, let us put $\tilde{d} := d_{\gamma}$ (then,  $\tilde{d}_{\gamma} = d_{2\gamma}$). Let $\mu \in \Free(M, \tilde{d}_{\gamma})$ be a norm-one finitely supported element having an optimal representation and write 
\begin{equation*}
    \mu = \sum_{i=1}^n \lambda_i \frac{\delta(x_i) - \delta(y_i)}{\tilde{d}(x_i, y_i) + \gamma} \in \Free(M, \tilde{d}_{\gamma})
\end{equation*}
with $\sum_{i=1}^n \lambda_i = \|\mu\|_{\Free(M, \tilde{d}_{\gamma})} = 1$. Take $f \in S_{\Lip(M, \tilde{d}_{\gamma})}$ to be such that 
\begin{equation*} \label{dense-SSD-expression-1}
{f}(x_i) - f(y_i) = \tilde{d}(x_i, y_i) + \gamma, \ \forall \ i=1,\ldots, n.
\end{equation*}
Arguing as in the proof of Theorem \ref{main1}, observe that the sets $\{x_i\}_{i=1}^n$ and $\{y_i\}_{i=1}^n$ are disjoint. 
 Now, we consider the finite set $N:=\{x_1, y_1, \ldots, x_n, y_n \}$ and assume, without loss of generality, that $0 \in \{y_1,\ldots, y_n\}$.

 Define $\tilde{f}:N \rightarrow \R$ by $\tilde{f}(x_i):=f(x_i) - \gamma$ and $\tilde{f}(y_i):=f(y_i)$ for every $i=1,\ldots, n$. We claim that $\tilde{f}$ is $1$-Lipschitz in $\Lip(N, \tilde{d})$. For this, let $i,j \in \{1,\ldots, n\}$. We split the argument into two cases. Suppose first that $\tilde{f}(x_i) - \tilde{f}(y_j) \geq 0$. Then, we have that 
\begin{equation*}
    |\tilde{f}(x_i) - \tilde{f}(y_j)| = f(x_i) - f(y_j) - \gamma \leq \tilde{d}(x_i, y_j)
\end{equation*}
and we are done. Let us suppose now that $\tilde{f}(x_i) - \tilde{f}(y_j) < 0$. Then 
\begin{align*}
|\tilde{f}(x_i) - \tilde{f}(y_j)| = \tilde{f}(y_j) - \tilde{f}(x_j) + (\tilde{f}(x_j) - \tilde{f}(x_i)) &= - \tilde{d}(y_j, x_j) + f(x_j) - f(x_i) \\
&\leq - \tilde{d}(y_j, x_j) + \tilde{d}(x_j, x_i) + \gamma \stackrel{(*)}{\leq} \tilde{d}(x_i, y_j),
\end{align*}
where the last inequality ($*$) holds by the definition of $\tilde{d} = d_\gamma$.

According to McShane's theorem, we can extend $\tilde{f}$ to a norm one element in $\Lip (M, \tilde{d})$. Let us denote this extension once again by $\tilde{f}$.  By Theorem \ref{main1}, applied to the metric space $(M, \tilde{d})$ and $\mu \in \mathcal{F}(M, \tilde{d}_\gamma)$, we can conclude that $\mu$ is an SSD point in $\Free(M, \tilde{d}_{\gamma})$ (noting that $\tilde{f}(x_i)-\tilde{f}(y_i) = \tilde{d}(x_i,y_i)$ for every $i=1,\ldots, n$). In other words, given $\mu \in \Free(M, d_{2\gamma})$ with an optimal representation, we proved that $\mu$ is an SSD point in $\Free(M, d_{2\gamma})$. 

Since $\gamma > 0$ was chosen arbitrarily, we conclude the result, and the denseness result follows directly. 
\end{proof}

\begin{remark}
    Suppose that we start with the metric space $(M, d)$ and $\gamma >0$ such that $d_{-\gamma} : M \times M \rightarrow \mathbb{R}$ given by $d_{-\gamma} (x,y) := d(x,y) - \gamma$ for $x \neq y \in M$ and $d_{-\gamma}(x,x):=0$ is a metric on $M$. Then Corollary \ref{main} yields that every finitely supported element in $\mathcal{F}(M,d)$ is an SSD point. However, the assumption that $d_{-\gamma}$ is a metric imposes that the original metric space $(M, d)$ satisfies $\inf_{z \in M \setminus \{x, y\}} G_z(x, y) \geq \gamma$ for $x \neq y \in M$ (i.e., $(M, d)$ is locally uniformly non-aligned with uniform constant $\gamma$). Additionally, $(M, d)$ must be uniformly discrete since $d(x, y) > \gamma$ for every $x \neq y \in M$.   
\end{remark}

Note that if $(M,d)$ is uniformly discrete, then the identity map between $(M,d)$ and $(M, d_\gamma)$ turns to be bi-Lipschitz for every $\gamma>0$. In this regard, the following result is a straightforward consequence of Corollary \ref{main}.

\begin{corollary}\label{cor:biLipschitz}
    Let $M$ be a uniformly discrete metric space and $\eps >0$ be given. Then there exists a metric space $N$ and a bi-Lipschitz map $\phi : M \rightarrow N$
    with {$1\leq \|\phi\|_{\text{\rm Lip}} \leq 1+\eps$} such that the set of SSD points in $\Free (N)$ is dense. 
\end{corollary}
 
\vspace{0.2cm} 
 
\noindent 
\textbf{Acknowledgements:} This paper was done while the first, the second and the fourth authors were visiting the Department of Mathematics of the Faculty of Electrical Engineering at the Czech Technical University in Prague. They are thankful for the support they have received there.  They would also like to thank Ramón J. Aliaga, Rubén Medina, Rafael Payá and Andrés Quilis for fruitful conversations and for providing relevant references about the topic of the paper, as well as anonymous referees for their very careful reading of our manuscript and the valuable suggestions.


\vspace{0.2cm} 
 \noindent 
\textbf{Funding information}: 

\vspace{0.2cm}
\noindent
{\it Christian Cobollo} was supported by 
\begin{itemize}
\itemsep0.25em
\item[(1)] The grants PID2021-122126NB-C33, PID2019-105011GB-I00 and PID2022-139449NB-I00 funded by 
MICIU/AEI/10.13039/501100011033 and by ERDF/EU,
\item[(2)] Generalitat Valenciana (through Project PROMETEU/2021/070 and the predoctoral contract CIACIF/2021/378). 

\end{itemize}

\vspace{0.2cm}
\noindent
{\it Sheldon Dantas} was supported by 
\begin{itemize} 
\itemsep0.25em
\item[(1)] The Spanish AEI Project PID2019-106529GB-I00/AEI/10.13039/501100011033,
\item[(2)]  Generalitat Valenciana project CIGE/2022/97 and 
\item[(3)] The grant PID2021-122126NB-C33 funded by 
MICIU/AEI/10.13039/501100011033 and by ERDF/EU.
\end{itemize} 
\vspace{0.2cm}
\noindent
{\it Petr Hájek} was supported by 
\begin{itemize}
\itemsep0.25em
    \item[(1)] GA23-04776S and 
    \item[(2)] SGS24/052/OHK3/1T/13 of CTU in Prague.
\end{itemize} 

\vspace{0.2cm}
\noindent
{\it Mingu Jung} was supported by 
June E Huh Center for Mathematical Challenges (HP086601) at Korea Institute
for Advanced Study.

 \end{document}